\documentclass{amsart}
\usepackage{amssymb,amsmath, amsfonts, amsrefs, tikz, epsfig, float}
\usepackage{amsthm}
\usepackage{mathrsfs}
\usepackage{enumitem, indentfirst}
\usepackage[fleqn,tbtags]{mathtools}
\usepackage{color}
 \newtheorem{theorem}{Theorem}[section]

 \newtheorem{corollary}[theorem]{Corollary}
 \newtheorem{lemma}[theorem]{Lemma} 
 \newtheorem{proposition}[theorem]{Proposition}

 \theoremstyle{definition}

 \theoremstyle{remark}

 \newtheorem*{remark}{Remark}
  \numberwithin{equation}{section}

\renewcommand{\theta}{\vartheta}

\DeclareMathOperator{\tform}{\mathfrak{t}}

\DeclareMathOperator{\wform}{\mathfrak{w}}

\DeclareMathOperator{\ran}{ran}

\DeclarePairedDelimiterX\sipt[2]{\langle}{\rangle^{}_{T}}{#1\,\delimsize\vert\,#2}
\DeclarePairedDelimiterX\sips[2]{\langle}{\rangle^{}_{S}}{#1\,\delimsize\vert\,#2}

\DeclarePairedDelimiterX\sipv[2]{(}{)_{v}}{#1\,\delimsize\vert\,#2}
\DeclarePairedDelimiterX\sipw[2]{(}{)_{w}}{#1\,\delimsize\vert\,#2}

\newcommand{\dupN}{\mathbb{N}}
\newcommand{\seq}[1]{(#1_{n})_{n\in\dupN}}

\newcommand{\dupR}{\mathbb{R}}
\newcommand{\dupC}{\mathbb{C}}
\newcommand{\dupK}{\mathbb{K}}

\newcommand{\dom}{\operatorname{dom}}

\newcommand{\dt}{\mathcal{D}_*(T)}
\newcommand{\dtt}{\mathcal{D}_*(T^*T)}

\newcommand{\D}{\mathcal{D}}

\newcommand{\hil}{\mathcal{H}}

\newcommand{\ehil}{\mathcal{E}}

\newcommand{\kil}{\mathcal{K}}
\newcommand{\bh}{\mathcal{B}(\hil)}

\DeclarePairedDelimiterX\abs[1]{\lvert}{\rvert}{#1}
\DeclarePairedDelimiterX\sip[2]{\langle}{\rangle}{#1,#2}
\DeclarePairedDelimiterX\siptilde[2]{(}{)_{\!_{\widetilde{A}}}}{#1\,\delimsize\vert\,#2}
\DeclarePairedDelimiterX\sipf[2]{(}{)_{f}}{#1\,\delimsize\vert\,#2}
\DeclarePairedDelimiterX\sipg[2]{(}{)_{g}}{#1\,\delimsize\vert\,#2}
\DeclarePairedDelimiterX\siptw[2]{(}{)_{\tform+\wform}}{#1\,\delimsize\vert\,#2}
\DeclarePairedDelimiterX\set[2]{\{}{\}}{#1\,:\,#2}
\DeclarePairedDelimiterX\dual[2]{\langle}{\rangle}{#1,#2}
\DeclarePairedDelimiterX\sipa[2]{(}{)_{\!_A}}{#1\,\delimsize\vert\,#2}
\DeclarePairedDelimiterX\sipe[2]{\langle}{\rangle_{\!_{\mathcal E}}}{#1\,\delimsize\vert\,#2}
\DeclarePairedDelimiterX\sipc[2]{(}{)_{\!_C}}{#1\,\delimsize\vert\,#2}
\DeclarePairedDelimiterX\sipab[2]{(}{)_{\!_{A+B}}}{#1\,\delimsize\vert\,#2}
\DeclarePairedDelimiterX\sipb[2]{(}{)_{\!_B}}{#1\,\delimsize\vert\,#2}

\newcommand{\limn}{\lim\limits_{n\rightarrow\infty}}

\newcommand{\opmatrix}[4]{\begin{bmatrix} #1 &  #2 \\   #3& #4\end{bmatrix}}

\allowdisplaybreaks

\begin{document}
\title[Extensions of positive operators]{Extensions of positive symmetric operators and Krein's uniqueness criteria}

\author[Z. Sebesty\'en]{Zolt\'an Sebesty\'en}
\author[Zs. Tarcsay]{Zsigmond Tarcsay}
\thanks{The author was supported by   the J\'anos Bolyai Research Scholarship of the Hungarian Academy of Sciences, and by the \'UNKP--21-5-ELTE-1090 New National Excellence Program of the Ministry for Innovation and Technology. Project no. TKP2021-NVA-09 has been implemented with the support
provided by the Ministry of Innovation and Technology of Hungary from
the National Research, Development and Innovation Fund, financed under
the TKP2021-NVA funding scheme. ``Application Domain Specific Highly Reliable IT Solutions'' project  has been implemented with the support provided from the National Research, Development and Innovation Fund of Hungary, financed under the Thematic Excellence Programme TKP2020-NKA-06 (National Challenges Subprogramme) funding scheme.}
\address{%
Zs. Tarcsay \\ Department of Applied Analysis  and Computational Mathematics\\ E\"otv\"os Lor\'and University\\ P\'azm\'any P\'eter s\'et\'any 1/c.\\ Budapest H-1117\\ Hungary }
\email{zsigmond.tarcsay@ttk.elte.hu}
\address{%
Z. Sebesty\'en \\ Department of Applied Analysis  and Computational Mathematics\\ E\"otv\"os Lor\'and University\\ P\'azm\'any P\'eter s\'et\'any 1/c.\\ Budapest H-1117\\ Hungary }
\email{sebesty@cs.elte.hu}



\dedicatory{Dedicated to the memory of John von Neumann (1903-1957)}

\begin{abstract}
We revise Krein's extension theory of positive symmetric operators. Our approach using factorization through an auxiliary Hilbert space has several advantages: it can be applied to non-densely defined transformations and it works in both real and complex spaces. As an application of the results and the construction we consider positive self-adjoint extensions of the modulus square operator $T^*T$ of a densely defined linear transformation $T$ and  bounded self-adjoint extensions of a symmetric operator. Krein's results on the uniqueness of positive (respectively, norm preserving) self-adjoint extensions are also revised. 
\end{abstract}

\maketitle

\section{Introduction}
The complete description of positive self-adjoint extensions of a densely defined positive operator $T$ (acting in a complex Hilbert space $\hil$) was done in the seminal work \cites{Krein12} of M. G. Krein. Krein not only showed that such an operator always has a positive self-adjoint extension, but that these extensions form an (unbounded) operator interval $[T_N,T_F]$ with respect to the form order, also introduced by Krein. The smallest element $T_N$ of that interval is called the Krein-von Neumann extension, while the largest one $T_F$ is known as the Friderichs extension. In this paper, our primary goal is to revise Krein's classical results (including his uniqueness criteria) and extend the extension theory to operators that are not necessarily densely defined. It is also a novelty that our procedure does not use the spectral theory of symmetric operators, so all our results remain valid in \textit{real} Hilbert spaces. 

The factorization procedure we use goes back to the article by the first author and J. Stochel \cites{SebStoch91}. Taking advatage of that treatment we revise Krein's uniqueness criterion \cite{Krein12} (cf. also \cite{HassiSandoviciDeSnoo}*{Theorem 4.7} for the case of positive linear relations). An additional benefit of generalizing Krein's results to non-densely defined operators will be apparent when considering the positive self-adjoint solvability of the operator equation $XA=B$ (see Corollary \ref{C:opequation} below). Another immediate application is the interesting fact that the `modulus square' operator $T^*T$ of an arbitrary densely defined operator $T$ always has a positive self-adjoint extension. (Surprisingly, this is not the case with $TT^*$: it can be even non-closable, see \cite{SebTarcsayWasaa}).

In the second half of our article, we deal with bounded extensions. First, in Theorem \ref{T:Neumannext} we provide a refinement of \cite{Sebestyen83a}*{Theorem 1}  characterizing positive operators having bounded positive extensions.  With the help of that result, a simple proof can be given for the existence of the transformation called 'shorted operator' by Krein \cite{Krein12} (see also \cites{Anderson1, Anderson2, PekarevSmulian}). On the other hand, Theorem \ref{T:Neumannext} also enables us to investigate the norm preserving  self-adjoint extensions of bounded symmetric operators  and  to revise Krein's uniqueness condition. The main ingredient in our proof is a formula describing  the range space of the square root of the shorted operator (Theorem \ref{T:ranS-TN}).

The following notations will be used throughout the paper. Let $\hil$ and $\kil$ be a real or complex Hilbert space and let $T:\hil\to\kil$ be a linear operator. In this last statement, we mean that the domain of $T$ (in notation: $\dom T$) is a linear subspace of $\hil$, while the range space (in notation: $\ran T$) is a (linear) subspace of $\kil$. If $\hil=\kil$ and $T$ satisfies 
\begin{equation*}
    \sip{Tf}{g}=\sip{f}{Tg},\qquad f,g\in\dom T,
\end{equation*}
then  $T$ is called  \textit{symmetric}. If in addition the quadratic form of $T$ is non-negative, that is, 
\begin{equation*}
    \sip{Tf}{f}\geq 0, \qquad f\in\dom T,
\end{equation*}
then $T$ is called \textit{positive}. If the underlying Hilbert space $\hil$ is complex then every operator $T$ having real quadratic form is automatically symmetric. 
Nevertheless, in this note we do not restrict ourselves to complex spaces, so  when speaking about positive operators, we always assume the symmetry of the transformation in question.

We recall the adjoint of a densely defined operator $T$ which is defined on its domain
\begin{equation*}
    \dom T^*\coloneqq \set{k\in\kil}{(\forall f\in\dom T)\,:\, \sip{Tf}{k}=\sip{f}{k^*}\quad \mbox{for some $k^*\in\hil$}}
\end{equation*}
by letting
\begin{equation*}
    T^*k\coloneqq k^*.
\end{equation*}
As it is well known, a densely defined  operator $T:\hil\to\hil$  is symmetric if and only if $T\subset T^*$, i.e., $T^*$ extends $T$. We call $T$ \textit{self-adjoint} if $T$ is densely defined and $T^*=T$. Recall also that the linear operator $T:\hil\to\kil$ is closed if its graph
\begin{equation*}
    G(T)\coloneqq \set{(f,Tf)}{f\in\dom T}
\end{equation*}
is a closed linear subspace of the product Hilbert space $\hil\times\kil$. $T$ is closable if it has a closed extension. The minimal closed extension of a closable operator $T$ is denoted by $\bar T$ and its graph is given by
\begin{equation*}
    G(\bar T)=\overline{G(T)}.
\end{equation*} 
The adjoint of a densely defined operator is always closed. In particular, every self-adjoint operator is closed. 
It is also well known that the closure of a densely defined operator $T$ is just its second adjoint $T^{**}$.

If $T$ is closed then a  linear subspace $\D$ of $\dom T$ is called a \textit{core for} $T$ if 
\begin{equation*}
    \overline{G(T|_{\D})}=G(T).
\end{equation*}
Throughout the paper we shall frequently use a fundamental theorem due of J. von Neumann \cite{Neumann} which says that $T^*T$ and $TT^*$ are both positive and self-adjoint operators provided that $T:\hil\to\kil$ is
densely defined and closed. In that case, $\D\coloneqq \dom T^*T$ is a core for $T$. The unique  square root of a positive self-adjoint operator $A$ will be denoted by $A^{1/2}$. If $T$ is densely defined and closed, then the domain of $(T^*T)^{1/2}$ is identical with $\dom T$. We mention here that the existence of the square root can be easily verified using the spectral theorem. An elementary treatment that also applies for real Hilbert spaces can be found in \cite{SebTarcs2016}.
 


 
\section{Positive self-adjoint extensions of positive symmetric operators}\label{S:section2}
 
Let $\hil$ be a real or complex Hilbert space. In what follows we fix a positive and symmetric operator $T:\hil\to\hil$ whose domain $\dom T$ is a linear subspace of $\hil$. We do not assume $\dom T$ to be dense or closed. Our first result below provides  various sufficient and necessary conditions on $T$ under which it extends to a positive and self-adjoint operator $\widetilde T$. The cornerstone of those characterizations is the linear subspace 
\begin{equation}\label{E:D(T)}
    \D_*(T)\coloneqq \set{g\in\hil}{\sup\set{\abs{\sip{Th}{g}}^2}{h\in\dom T, \sip{Th}{h}\leq 1}<+\infty}.
\end{equation}
Clearly, $g\in\D_*(T)$ holds if and only if there is a constant $m_g\geq 0$ (depending only on $g$) such that 
\begin{equation*}
    \abs{\sip{Th}{g}}^2\leq m_g\sip{Th}{h},\qquad (\forall h\in\dom T).
\end{equation*}
From the Cauchy-Schwarz inequality applied to the form $(h,k)\mapsto \sip{Th}{k}$ it follows that 
\begin{equation}\label{E:domTreszeDT}
    \dom T\subseteq \D_*(T).
\end{equation}

\begin{theorem}\label{T:Theorem1}
Let $T:\hil\to\hil$ be a positive symmetric operator. Then the following statements are equivalent:
\begin{enumerate}[label=\textup{(\roman*)}]
    \item $T$ has a positive self-adjoint extension,
    \item $\dt\subseteq \hil$ is dense,
    \item $\dt^\perp\subseteq \ran(I+T)$,
    \item $\dt^\perp\cap \overline{\ran T}=\{0\}$,
    \item For every sequence $\seq h$ of $\dom T$ such that $\sip{Th_n}{h_n}\to0$ and $Th_n\to f$ it follows that $f=0$,
    \item There exists a Hilbert space $\mathcal E$ and a densely defined linear operator $V:\hil\to \mathcal E$ with $\dom V\supset \dom T$ such that $V(\dom T)^\perp=\{0\}$ and
    \begin{equation}\label{E:VgVh}
        \sipe{Vg}{Vh}=\sip{g}{Th},\qquad  g\in\dom V,h\in\dom T.
    \end{equation}
\end{enumerate}
\end{theorem} 
\begin{proof}
(i)$\Rightarrow$(ii): Let  $\widetilde{T}$ be any positive self-adjoint extension of $T$ then it is easy to see that 
\begin{equation*}
    \dom \widetilde{T}^{1/2}\subseteq \dt.
\end{equation*}
For if $g\in\dom \widetilde{T}^{1/2}$ then  it follows that 
\begin{equation*}
    \abs{\sip{Th}{g}}^2=\abs{\sip{\widetilde{T}^{1/2}h}{\widetilde{T}^{1/2}g}}^2\leq \|\widetilde{T}^{1/2}g\|^2\sip{Th}{h},\qquad h\in\dom T,
\end{equation*}
whence $g\in\dt$, indeed.

(ii)$\Rightarrow$(iii): This implication is obvious.

(iii)$\Rightarrow$(iv): Take any vector $g\in\dt^\perp$ then there exists $h\in\dom T$ such that $g=h+Th$ by (iv). Then $h\in\dt$ as well and thus 
\begin{equation*}
    0=\sip gh=\sip{h}{h}+\sip{Th}{h},
\end{equation*}
which yields $h=0$ due to positivity of $T$.

(iv)$\Rightarrow$(v): Let $\seq h$ and $f\in\hil$ be as in (v). Clearly, $f\in\overline{\ran T}$. On the other hand, for every $g\in\dt$ we have
\begin{equation*}
    \abs{\sip{Th_n}{g}}^2\leq m_g\sip{Th_n}{h_n},
\end{equation*}
so letting $n\to+\infty$ implies $\sip{f}{g}=0$. Hence $f\in\dt^\perp$ and therefore $f=0$ by (iv).

(v)$\Rightarrow$(vi): Consider the following inner product $\sipe\cdot\cdot$ on the range space $\ran T$ of $T$:
\begin{equation*}
    \sipe{Tf}{Th}\coloneqq \sip{Tf}{h},\qquad f,h\in\dom T.
\end{equation*}
From (v) it follows that $\sipe\cdot\cdot$ is indeed a well defined inner product: for if $\sip{Th}{h}=0$ for some $h\in\dom T$ then the sequence $h_n\coloneqq h$ clearly satisfies the conditions of (v) with $f=Th$ hence $Th=0$. Let $\mathcal E_T$ denote the `\textit{energy space}' of $T,$ that is the completion of the  prehilbert space so obtained. Let us denote by $J_T$ the natural embedding of $\ran T\subseteq \mathcal{E}_T$ into $\hil$ defined via
\begin{equation}
    J_T(Th)\coloneqq Th,\qquad h\in\dom T.
\end{equation}
Clearly, $J_T:\mathcal E_T\to\hil$ is densely defined and  condition (v) expresses just that $J_T$ is closable. Furthermore, for every $f,h\in\dom T$ one has
\begin{equation*}
    \sip{J_T(Th)}{f}=\sip{Th}{f}=\sipe{Th}{Tf},
\end{equation*}
whence we conclude that $\dom T\subseteq \dom J_T^*$ and 
\begin{equation}\label{E:JT*}
    J_T^*f=Tf\in\mathcal E_T,\qquad f\in\dom T.
\end{equation}
As a consequence we see that $J_T^*(\dom T)=\ran T\subseteq \mathcal E_T$ is dense,
\begin{equation*}
    \sipe{J_T^*g}{J_T^*h}=\sipe{J_T^*g}{Th}=\sip{g}{Th},\qquad g\in\dom J_T^*, h\in\dom T,
\end{equation*}
which means that 
$V\coloneqq J_T^*$ fulfills every  condition of statement (vi).

(vi)$\Rightarrow$(i): Consider a linear operator $V:\hil\to\mathcal E$ satisfying all the properties stated in (vi). We are going to show that $\widetilde{T}\coloneqq V^*V^{**}$ is then a positive self-adjoint extension of $T$. With \eqref{E:VgVh} in mind, it will follow immediately if we show that  $V(\dom T)\subseteq \dom V^*$. Take therefore $g\in\dom V$ and $h\in\dom T$ then \begin{equation*}
    \sipe{Vg}{Vh}=\sip{g}{Th}.
\end{equation*}
This implies $Vh\in\dom V^*$. 
\end{proof}

\begin{remark}
Let us make some comments about Theorem \ref{T:Theorem1}:
\begin{enumerate}
    \item A densely defined positive symmetric operator $T$ always has a positive self-adjoint extension. This follows immediately from inclusion \eqref{E:domTreszeDT}.
    \item Condition (v) was called positive closability by Ando and Nishio in \cite{AndoNishio}.
    \item The equivalence between (i) and (ii) was established in \cite{SebStoch91}*{Theorem 1}. Conditions (iii), (iv) and (vi) are refinements of that result.
    \item An immediate calculation shows that 
    \begin{equation}\label{E:DT=domJT}
        \D_*(T)=\dom J_T^*,
    \end{equation}
    whence one obtains
    \begin{equation*}
        \D_*(T)=\dom(J^{**}_TJ^*_T)^{1/2},
    \end{equation*}
    where $A^{1/2}$ denotes the  square root of the positive self-adjoint operator $A$.
\end{enumerate}
\end{remark}

A transformation having a positive and self-adjoint extension must apparently be  closable. For this reason, from the point of view of self-adjoint extendibility, it is not a serious restriction if we assume the operator in question to be closable. Under this additional assumption condition (iii) of Theorem \ref{T:Theorem1} can be weakened as follows:
\begin{corollary}
Assume that the positive symmetric operator  $T$ is closable. Then the following condition is still equivalent with the conditions (i)-(vi) of Theorem \ref{T:Theorem1}:
\begin{itemize}
    \item[\textup{(vii)}] $\dt^\perp\subseteq \overline{\ran (I+T)}$. 
\end{itemize}
\end{corollary}
\begin{proof}
It is clear that (vii) is formally weaker than condition (iv). Hence it is enough to prove that a closable positive operator $T$ which fulfills (vii), has a positive self-adjoint extension. To do so we prove first that 
\begin{equation*}
    \D_*(T)=\D_*(\bar T), 
\end{equation*}
where $\bar T$ denotes the closure of $T$. It is clear that $\D_*(\bar T)\subseteq \D_*(T)$. For the converse, let $g\in\D_*(T)$ and $h\in\dom \bar T$, and choose a $\seq h$ from $\dom T$ such that $h_n\to h$, $Th_n\to \bar Th$. Then 
\begin{align*}
    \abs{\sip{\bar Th}{g}}^2=\limn\abs{\sip{Th_n}{g}}^2\leq \limn m_g\cdot \sip{Th_n}{h_n}=m_g \cdot \sip{\bar Th}{h},
\end{align*}
hence  $g\in \D_*(\bar T)$, as it is claimed. 

We prove now that (vii) implies (i). It is clear that $T$ has a positive and self-adjoint extension if and only if $\bar T$ does. By Theorem \ref{T:Theorem1}, this is equivalent to inclusion
\begin{equation}\label{E:D(Tclosure)}
    \D_*(\bar T)\subseteq \ran(I+\bar T).
\end{equation}
As we saw above, $\D_*(\bar T)=\D_*(T)$. On the other hand, $I+\bar T$ is a closed and bounded below  operator, thus we have
\begin{equation*}
    \ran (I+\bar T)=\overline{\ran (I+\bar T)}=\overline{\ran (I+T)}.
\end{equation*}
The equivalence between (vii) and \eqref{E:D(Tclosure)} is now obvious. 
\end{proof}
As an immediate application of Theorem \ref{T:Theorem1} the positive self-adjoint solvability of operator equations of  type
\begin{equation*}
    XA\supseteq B
\end{equation*}
is considered in the following result: 

\begin{corollary}\label{C:opequation}
Let $\hil$ and $\kil$ be real or complex Hilbert spaces and let  $A,B:\kil\to\hil$ be (not necessarily densely defined or closable) linear operators such that $\dom B\subset \dom A$. The following statements are equivalent:
\begin{enumerate}[label=\textup{(\roman*)}]
    \item there exists a positive self-adjoint operator $S$ such that 
 \begin{equation*}
     SA\supseteq B,
 \end{equation*}   
 \item $\sip{Bh}{Ah}\geq 0$ for every $h\in\dom B$ and the set of those vectors $g$ such that 
 \begin{equation*}
     \sup\set{\abs{\sip{Bh}{g}}}{h\in\dom B, \sip{Bh}{Ah}\leq 1}<+\infty
 \end{equation*}
 is dense in $\hil$,
 \item for every sequence $\seq h$ of $\dom B$ such that $\sip{Bh_n}{Ah_n}\to0$ and $Bh_n\to f$ it follows that $f=0$.
\end{enumerate}
\end{corollary}
\begin{proof} Every operator $S$ that satisfies $SA\subseteq B$ is an extension of the  operator
\begin{equation*}
T:\ran A\to\hil, \qquad T(Ax)\coloneqq Bx.    
\end{equation*}
Hence, (i) is equivalent to the positive self-adjoint extendibility of $T$. The equivalence between (i)-(iii) follows from Theorem \ref{T:Theorem1} in a straightforward way.  
\end{proof}
\begin{remark}
We emphasize that in the proof of Corollary \ref{C:opequation}, we made significant use of the main advantage of  Theorem \ref{T:Theorem1}, according to which the operator $T$ under consideration does not have to be densely defined.  
\end{remark}

\section{The Krein-von Neumann extension}
Let $T:\hil\to\hil$ be a (not necessarily densely defined) positive symmetric operator that fulfills any (hence all) of the equivalent conditions of Theorem \ref{T:Theorem1}. Then the operator $J_T:\mathcal E_T\to\hil$ is densely defined and closable and it turned out from the proof of Theorem \ref{T:Theorem1} that 
\begin{equation*}
    T_N\coloneqq J_T^{**}J_T^*
\end{equation*}
is a positive self-adjoint extension of $T$. In this section we are going to show that $T_N$ is the smallest among all positive self-adjoint extensions of $T$, hence we shall call it the \textit{Krein-von Neumann extension} of $T$ (cf. \textbf{cccc}). The minimality of $T_N$ is   understood with respect to the so-called \textit{form order} `$\preceq$', which is a partial ordering among the set of all positive self-adjoint operators, defined by
\begin{equation}\label{E:formorder}
    S_1\preceq S_2\quad \overset{\mbox{def}}{\iff}\quad (I+S_2)^{-1}\leq (I+S_1)^{-1}.
\end{equation}
It can be proved that 
\begin{equation*}
    S_1\preceq S_2\quad \iff\quad \begin{cases}\dom S_2^{1/2}\subseteq \dom S_1^{1/2},\\
    \|S_1^{1/2}g\|^2\leq \|S_2^{1/2}g\|^2,\qquad g\in\dom S_2^{1/2},
    \end{cases}
\end{equation*}
wher $S_i^{1/2}$ stands for the unique positive self-adjoint square root of $S_i$ (see eg. \cite{SebTarcs2017}). 

Let now $S$ be a positive self-adjoint extension of $T$; we are going to prove that 
\begin{equation*}
    T_N\preceq S.
\end{equation*}
An easy calculation shows that $\dom J_T^*=\dt$ from which one concludes the identity
\begin{equation*}
    \dom T_N^{1/2}=\dom( J_T^{**}J_T^*)^{1/2}=\dt.
\end{equation*}
Furthermore, from the density of $\ran T$ in the energy space $\mathcal E_T$ (with respect to the inner product $\sipe{\cdot}{\cdot}$)
it follows that 
\begin{align*}
    \|T_N^{1/2}g\|^2&=\sipe{J_T^*g}{J_T^*g}\\
    &=\sup\set{\abs{\sipe{J_T^*g}{Th}}^2}{h\in\dom T,\sipe{Th}{Th}\leq 1}.
\end{align*}
Hence, using the identity 
\begin{equation*}
    \sipe{J_T^*g}{Th}=\sip{g}{Th}
\end{equation*}
we gain the useful formula
\begin{equation}\label{E:3.1}
    \|T_N^{1/2}g\|^2=\sup\set{\abs{\sip{g}{Th}}^2}{h\in\dom T,\sip{Th}{h}\leq 1}.
\end{equation}
Recall that the self-adjoint operator $S$ does not have any proper self-adjoint extension. Hence, using the above procedure with $S$ instead of $T$, we obtain that  $ S=S_N$ and also that  $$\dom S^{1/2}=\mathcal{D}_*(S)\subseteq \dt.$$ 
In particular, equality \eqref{E:3.1} applies to $S=S_N$ and gives
\begin{align*}
    \|S^{1/2}g\|^2&=\sup\set{\abs{\sip{g}{Sh}}^2}{h\in\dom S,\sip{Sh}{h}\leq 1}\\
    &\geq \sup\set{\abs{\sip{g}{Th}}^2}{h\in\dom T,\sip{Th}{h}\leq 1}\\
    &=\|T_N^{1/2}g\|^2,
\end{align*}
which means that $$T_N\preceq S.$$ 

What has just been proven can be summarized in the following result:
\begin{theorem}\label{T:KreinNeumann}
If the positive symmetric operator $T:\hil\to\hil$ has any positive self-adjoint extension (i.e., $T$ fulfills any of the equivalent conditions of Theorem 1) then $T_N\coloneqq J_T^{**}J_T^*$ is the smallest positive self-adjoint extension of $T$.
\end{theorem}

\bigskip

In what follows we are going to establish the converse of Theorem \ref{T:KreinNeumann}. Namely, it will be proved that  a positive self-adjoint operator $S$ satisfing $T_N\preceq S$ mus be an extension of  $T$. The precise statement is given in the next result:
\begin{theorem}\label{T:Krein-converse}
Assume that the linear operator $T$ (acting in the real or complex Hilbert space $\hil$) satisfies the equivalent conditions of Theorem \ref{T:Theorem1}. For a given positive self-adjoint operator $S$ the following statements are equivalent:
\begin{enumerate}[label=\textup{(\roman*)}]
    \item $S$ is an extension of $T$,
    \item \begin{enumerate}[label=\textup{(\alph*)}]
        \item $T_N\preceq S$,
        \item $\dom T\subseteq \dom S^{1/2}$,
        \item $\|S^{1/2}f\|^2\leq \sip{Tf}{f}$ for every $f\in\dom T$.
    \end{enumerate}
\end{enumerate}
\end{theorem}
\begin{proof}
Assume first that $S$ is a positive self-adjoint extension of $T$. Then $T_N\preceq S$ by Theorem \ref{T:KreinNeumann}. On the other hand, $\dom T\subseteq \dom S\subseteq \dom S^{1/2}$ and 
\begin{equation*}
    \sip{Tf}{f}=\sip{Sf}{f}=\|S^{1/2}f\|^2
\end{equation*}
for every $f\in\dom T$.

To prove the converse direction suppose that the positive self-adjoint operator $S$ satisfies conditions (ii) (a)-(c). Introduce the semi inner product $p$ on $\dom S^{1/2}$ as 
\begin{equation*}
    p(f,g)\coloneqq \sip{S^{1/2}f}{S^{1/2}g}-\sip{T_N^{1/2}f}{T_N^{1/2}g},\qquad f,g\in\dom S^{1/2}.
\end{equation*}
By (a) and (c) we have $p(f,f)=0$ for $f\in\dom T$, thus 
\begin{equation*}
    p(g,f)=0,\qquad f\in\dom T, g\in\dom S^{1/2},
\end{equation*}
according to the Cauchy-Schwarz inequality. Consequently, 
\begin{align*}
\sip{S^{1/2}g}{S^{1/2}f}=\sip{T_N^{1/2}g}{T_N^{1/2}f}=\sip{g}{Tf}
\end{align*}
for every $f\in\dom T$ and $g\in\dom S^{1/2}$. From this last identity we see  that $S^{1/2}f\in\dom S^{1/2}$ (or equivalently, $f\in\dom S$) and also  $Sf=S^{1/2}(S^{1/2}f)=Tf$. This in turn means that $S$ is an extension of $T$.
\end{proof}
As an immediate consequence we obtain the following statement (see \cite{SebStoch2007}*{Corollary 10}):
\begin{corollary}\label{C:Cor10}
Let $T:\hil\to\hil$ be a positively closable symmetric operator and let $R,S:\hil\to\hil$ be positive self-adjoint operators. Suppose that $T\subset R$ and $T_N\preceq S\preceq R$ then also $T\subset S$.  
\end{corollary}
\begin{proof}
From the assumptions it follows that 
\begin{equation*}
    \dom T\subseteq \dom R\subseteq \dom R^{1/2}\subseteq \dom S^{1/2},
\end{equation*}
and also  
\begin{equation*}
    \sip{Tf}{f}=\|R^{1/2}f\|^2\geq \|S^{1/2}f\|^2
\end{equation*}
for every $f\in\dom T$. By Theorem \ref{T:Krein-converse} we conclude that $T\subset S$.
\end{proof}

We close the section by providing a necessary and sufficient condition on a positive self-adjoint operator $S$ to be equal to the Krein-von Neumann extension of $T$:
\begin{theorem}
Let $T$ be a (not necessarily densely defined) positive operator in the Hilbert space $\hil$ and let $S$ be a positive self-adjoint extension of $T$. The following statements are equivalent:
\begin{enumerate}[label=\textup{(\roman*)}]
    \item $T_N=S$,
    \item $\ran T_N^{1/2}=\ran S^{1/2}$.
\end{enumerate}
\end{theorem}
\begin{proof}
Let  $S$ be a self-adjoint extension of $T$ satifying range identity (ii). Consider the energy space $\mathcal E_S$ associated with $S$. We prove first that $\ran T$ is dense in $\mathcal E_S$ (with respect to the scalar product $\sips{\cdot}{\cdot}$ induced by $S$). To do so it suffices to prove that $\dom J_S^{**}\subseteq \overline{\ran T}$ where the closure is taken with respect to $\sips{\cdot}{\cdot}$. Consider a vector $\xi\in\mathcal \dom J_S^{**}$. Then $J_S^{**}\xi\in \ran S^{1/2}=\ran T_N^{1/2}$, hence there is a sequence $\seq f$ in $\dom T$ such that 
\begin{equation*}
    \sip{T(f_n-f_m)}{f_n-f_m}\to 0\qquad\mbox{and}\qquad Tf_n\to J^{**}_S\xi\in\hil.
\end{equation*}
Since we have $T\subset S$, it follows that $Tf_n=Sf_n$ and therefore  \begin{equation*}
    \sips{S(f_n-f_m)}{S(f_n-f_m)}\to0\qquad \mbox{and}\qquad J^{**}_{S}(Sf_n)\to J^{**}_S\xi.
\end{equation*}
Consequently, $Tf_n=Sf_n\to \zeta\in\mathcal{E}_S$, hence $\zeta\in \dom J_S^{**}$ and $J_S^{**}\zeta=J_S^{**}\xi$. By injectivity of $J_S^{**}$ we get $\xi=\zeta$ which in turn shows that $\xi$ belongs to the closure of $\ran T$ in $\mathcal E_S$. Since $\dom J_S^{**}$ is a dense subspace in $\mathcal E_S$, so is $\ran T$.

Let us denote by $J_0$ the restriction of $J_S^{**}$ to $\ran T$:
\begin{equation*}
    J_0\coloneqq J_S^{**}|_{\ran T}.
\end{equation*}
By the first part of the proof we see that $J_0:\mathcal{E}_S\to\hil$ is densely defined acting by 
\begin{equation*}
    J_0(Tf)=Tf,\qquad f\in\dom T.
\end{equation*}
It is easy to see that $J_0^{**}J_0^*$ is a positive self-adjoint extension of $T$. Besides, an argument similar to the calculation presented at the beginning of this section shows that 
\begin{equation*}
    \dom (J_0^{**}J_0^*)^{1/2}=\dom J_0^*=\dt=\dom J_T^*,
\end{equation*}
and that
\begin{equation*}
    \|(J_0^{**}J_0^*)^{1/2}g\|^2=\|(J_T^{**}J_T^*)^{1/2}g\|^{2},\qquad g\in\dt,
\end{equation*}
whence $J_0^{**}J_0^*=J_T^{**}J_T^*=T_N$. This identity shows that 
\begin{equation*}
    \ker J_0^{**}=\ker J_S^{**}=\{0\} \qquad \mbox{and}\qquad \ran J_0^{**}=\ran T_N^{1/2}=\ran J_S^{**}.
\end{equation*}
Since $J_0^{**}\subset J_S^{**}$, it follows that $J_0^{**}=J_S^{**}$ and therefore 
\begin{equation*}
    T_N=J_0^{**}J_0^*=J_S^{**}J_S^*=S,
\end{equation*}
as claimed.
\end{proof}

\section{The Friedrichs extension}
In his classical paper \cite{Friedrichs} K. Friedrichs proved that a \textit{densely defined} positive symmetric operator has at least one positive self-adjoint extension. In \cite{Krein12} Krein proved that the extension constructed by Friedrichs is the largest possible extension of $T$ (with respect to the form order \eqref{E:formorder}) and that the positive self-adjoint extensions of $T$ form an ``operator interval'' $[T_N,T_F]$. Here, $T_N$ is the Krein-von Neumann extension of $T$ while $T_F$ is the so called \textit{Friedrichs extension} of $T$ which will be investigated below in detail. 

In this section we are going to construct the largest extension $T_F$ in a way that is completely different from Friedrichs’ original approach. In fact, our procedure is more in line with the reasoning of \cite{Prokaj} and uses a factorization method through the energy space $\mathcal E_T$. However, the proof given here is somewhat shorter and simpler.  

It is easy to check that the set of positive self-adjoint extensions of a non-densely defined positive operator cannot have a largest element. For this reason, throughout the remainder of this section we assume that the positive operator $T:\hil\to\hil$ is densely defined. Recall that also the Krein-von Neumann extension of such an operator $T$  automatically exists according to Remark following Theorem \ref{T:Theorem1}.  

Keeping in mind the notations of Section \ref{S:section2} one concludes that 
\begin{equation*}
    \dom T\subseteq \dt=\dom J_T^*.
\end{equation*}
Hence the restriction $Q_T$ of $J^*_T$ to $\dom T$ is a densely defined (and necessarily closable) operator. By formula \eqref{E:VgVh}, $Q_T:\dom T\subseteq \hil\to \mathcal E_T$ acts by
\begin{equation}\label{E:QT*}
    Q_Tf=Tf,\qquad f\in\dom T.
\end{equation}
The closure $Q_T^{**}$ of $Q_T$ can be described as  
\begin{equation}
    \set{g\in\hil}{g_n\to g, \sip{T(g_n-g_m)}{g_n-g_m}\to 0 \quad \mbox{for some $(g_n)\subset \dom T$}}. 
\end{equation}
From  formulas $J_T^{**}\subset Q_T^*$ and \eqref{E:QT*} it follows that 
\begin{equation*}
    T_F\coloneqq Q_T^*Q_T^{**}
\end{equation*}
is a positive and self-adjoint extension of $T$. Our claim is to show that $T_F$ is equal to the largest extension of $T$, that is, $S\preceq T_F$ holds for every positive self-adjoint extension $S$ of $T$. 

For let $S$ be a positive self-adjoint extension of $T$. By repeating the above procedure with $S$ instead of $T$, we get that
\begin{equation}
    S=S_F=Q_S^*Q_S^{**}.
\end{equation}
As a consequence,  $\dom S^{1/2}=\dom Q_S^{**}$. Hence $S\supset T$ implies 
\begin{align*}
    \dom S^{1/2}&=\set{g\in\hil}{g_n\to g, \sip{S(g_n-g_m)}{g_n-g_m}\to 0 \quad \mbox{for some $(g_n)\subset \dom S$}}\\
    &\supseteq\set{g\in\hil}{g_n\to g, \sip{T(g_n-g_m)}{g_n-g_m}\to 0 \quad \mbox{for some $(g_n)\subset \dom T$}}\\
    &=\dom T_F^{1/2}.
\end{align*}
On the other hand, for $g\in\dom S^{1/2}$ we have 
\begin{equation*}
    \|S^{1/2}g\|^2=\lim_{n\to+\infty} \sip{Sg_n}{g_n} 
\end{equation*}
with $g_n\in\dom S$, $g_n\to g$ and $\sip{S(g_n-g_m)}{g_n-g_m}\to0$ because 
\begin{equation*}
     \|S^{1/2}g\|^2=\lim_{n\to+\infty} \sip{Q_S^{**}g_n}{Q_S^{**}g_n}_S=\lim_{n\to+\infty} \sip{Sg_n}{g_n}.
\end{equation*}
Similarly, for $g\in\dom T^{1/2}_F$ we have 
\begin{equation*}
    \|T_F^{1/2}g\|^2=\lim_{n\to+\infty} \sip{Tg_n}{g_n}=\lim_{n\to+\infty} \sip{Sg_n}{g_n}= \|S^{1/2}g\|^2
\end{equation*}
with $g_n\in\dom T$, $g_n\to g$ and $\sip{S(g_n-g_m)}{g_n-g_m}\to0$. Hence
\begin{equation}\label{E:TF12g}
    \|T_F^{1/2}g\|^2=\|S^{1/2}g\|^2,\qquad g\in\dom T_F^{1/2},
\end{equation}
and therefore $S\preceq T_F$, as it is claimed.

\medskip

With the above considerations we have just proved the following result:
\begin{theorem}\label{T:Friedrichs}
If $T:\hil\to\hil$ is a densely defined positive symmetric operator, then $T_F\coloneqq Q_T^{*}Q_T^{**}$ is the largest positive self-adjoint extension of $T$. The domain $\dom T_F^{1/2}$ of $T_F^{1/2}$ consists of those vectors $g$ for which there exists a sequence $(g_n)\subset \dom T$ such that $g_n\to g$ and $\sip{T(g_n-g_m)}{g_n-g_m}\to0$. In that  case,
\begin{equation*}
    \|T_F^{1/2}g\|^2=\limn\sip{Tg_n}{g_n}.
\end{equation*}
\end{theorem}
Combining Theorems \ref{T:KreinNeumann}, Theorem \ref{T:Friedrichs}, and Corollary \ref{C:Cor10} we conclude the following revised form of a fundamental result by Krein \cite{Krein12}: 
\begin{corollary}
The positive self-adjoint extensions of a densely defined positive self-adjoint operator $T$ (acting in a real or complex Hilbert space $\hil$) form an operator interval
\begin{equation}\label{E:TNTF}
    [T_N,T_F]=\set{S=S^*}{T_N\preceq S\preceq T_F},
\end{equation}
where $T_N=J_T^{**}J_T^*$ is the Krein-von Neumann extension, while $T_F=Q_T^{*}Q_T^{**}$ is the Friedrichs extension of $T$. 
\end{corollary}
\section{Krein's uniqueness  criterion}
In this section we are going to investigate the problem of uniqueness of the positive self-adjoint extensions. It is easy to check that uniqueness can appear only in the densely defined case. As we have seen in the preceding sections, positive self-adjoint extensions of a densely defined positive symmetric operator $T$ form an operator interval \eqref{E:TNTF},
where $T_N$ is the Krein-von Neumann extension, while $T_F$ is the Friedrichs extension of $T$. Taking into account of that, we see that $T$ has a unique positive self-adjoint extension if and only if its minimal and maximal extension coincide, i.e,  $T_N=T_F$. 

In the present section our main goal is to revise Krein's uniqueness condition. To do so we are going to present first a formula for a positive self-adjoint operator $S$ to agree with the Friedrichs extension of $T$:
\begin{theorem}\label{T:S=TF}
Let $T$ be a densely defined positive symmetric operator in the real or complex Hilbert space $\hil$ and let $S$ be a positive self-adjoint extension of $T$. Then the following statements are equivalent:
\begin{enumerate}[label=\textup{(\roman*)}]
    \item $S=T_F$,
    \item $\ker(I+T^*)\cap \dom S^{1/2}=\{0\}.$
\end{enumerate}
\end{theorem}
\begin{proof}
Before we start proving the desired equivalence, let us make a few observations. Consider the energy space $\mathcal{E}_S$ of $S$ and the linear operator $Q_S:\dom S\to \mathcal{E}_S$ satisfying
\begin{equation*}
    Q_Sf=Sf,\qquad f\in\dom S.
\end{equation*}
Let $Q_0$ denote the restriction of $Q_S$ to $\dom T$. Then $Q_0$ is closable and satisfies 
\begin{equation*}
    Q_0f=Tf,\qquad f\in\dom T,
\end{equation*}
and an easy calculation shows that $Q_0^*Q_0^{**}=T_F$. It follows therefore that  identity $T_F=S$ is equivalent to identity $Q_0^{**}=Q_S^{**}$.

Let us now turn to the proof of the equivalence between (i) and (ii). Assume  that $T_N\neq S$, or equivalently that $Q_0^{**}\subsetneqq Q_S^{**}$. Then there is a non-zero vector $g\in\dom Q_S^{**}=\dom S^{1/2}$ such that   $(g,Q_S^*g)$ in the graph of $Q_S^*$ is orthogonal to the graph of $Q_0^{**}$. Then 
\begin{align*}
    0&=\sip[\big]{(g,Q_S^*g)}{(f,Q_0f)}_{G(Q_S^{**})}\\
    &=\sip{g}{f}+\sipe{Q_S^*g}{Tf}\\
    &=\sip{g}{f}+\sip{g}{Tf}\\
    &=\sip{g}{f+Tf}
\end{align*}
for every  $f\in\dom T$.
Consequently, $g\in\ran(I+T)^{\perp}=\ker(I+T^*)$. Hence $\ker(I+T^*)\cap\dom S^{1/2}\neq\{0\}$ proving that (ii) implies (i). 

Let us assume now that $T_F=T_N$ and consider a vector $g\in\ker(I+T^*)\cap \dom S^{1/2}$. Then $g\in \dom T_F^{1/2}$ according to our hypothesis. By Theorem \ref{T:Friedrichs}  there exists a sequence $(g_n)$ from $\dom T$ such that 
\begin{equation*}
    g_n\to g\qquad \mbox{and}\qquad \sip{T(g_n-g_m)}{g_n-g_m}\to 0.
\end{equation*}
In particular,  $T_F^{1/2}g_n\to T_F^{1/2}g$ and thus
\begin{align*}
    -\|g\|^2&=\sip{T^*g}{g}=\limn \sip{T^*g}{g_n}\\&=\limn \sip{g}{Tg_n}=\limn \sip{T_F^{1/2}g}{T_F^{1/2}g_n}=\|T_F^{1/2}g\|^2\geq0,
\end{align*}
whence $g=0$. Consequently, (i) implies (ii).
\end{proof}
Using the preceding result we are able to establish the following generalization of Krein's uniqueness criterion \cite{Krein12} (cf. also \cite{HassiSandoviciDeSnoo}*{Theorem 4.7}):
\begin{corollary}\label{C:Krein-unique}
Let $T$ be a densely defined positive symmetric operator in the real or complex Hilbert space $\hil$. The following statements are equvalent:
\begin{enumerate}[label=\textup{(\roman*)}]
    \item $T$ has a unique positive self-adjoint extension, i.e., $T_N=T_F$,
    \item $\ker(I+T^*)\cap \dt=\{0\}$,
    \item for every non-zero vector $g\in \ker(I+T^*)$ one has
    \begin{equation*}
        \sup\set{\abs{\sip{f}{g}}^2}{f\in\dom T, \sip{Tf}{f}\leq 1}=+\infty.
    \end{equation*}
\end{enumerate}
\end{corollary}
\begin{proof}
Keeping in mind the identity $\dom T_N^{1/2}=\dt$, the equivalence between (i) and (ii) follows from the preceding theorem. Furthermore,  for a vector $g\in \ker (I+T^*)$ one has 
\begin{equation*}
    \abs{\sip{f}{g}}^2=\abs{\sip{f}{-T^*g}}^2=\abs{\sip{Tf}{g}}^2, \qquad f\in\dom T,
\end{equation*}
whence we see that $g\in\dom \dt$ if and only if the supremum in (iii) is finite. This proves the equivalence between (ii) and (iii).
\end{proof}

\bigskip

From all that we have seen so far it is clear that if a positive operator $T$ is essentially self-adjoint (i.e., $T^*=T^{**}$), then on has 
\begin{equation}\label{E:TN=TF}
    T_N=T_F.
\end{equation}
Nevertheless, from the equality \eqref{E:TN=TF}, it is still not clear whether $T$ could not have further `indefinite' self-adjoint extensions.
Based on our constructions of $T_N$ and $T_F$ we are going to prove that a bounded below positive operator $T$ satisfying \eqref{E:TN=TF} must be essentially self-adjoint. As a first step, we prove the following Lemma:
\begin{lemma}\label{L:Lemma}
Let $T$ be a densely defined positive symmetric operator, then
\begin{enumerate}[label=\textup{(\alph*)}]
\item $\dom(T_N)\cap\dom (T_F^{1/2})=\dom(T_N)\cap\dom (T_F)=\dom (J_T^{**}Q_T^{**})$,
\item $G(T_N)\cap G(T_F)=G(J_T^{**}Q_T^{**})$.
\end{enumerate}
\end{lemma}
\begin{proof}
Recall that $Q_T=J_T^*|_{\dom T}$ and that $T_N=J_T^{**}J_T^*$, while $T_F=Q_T^{*}Q_T^{**}$. Consider  $f\in\dom (J_T^{**}Q_T^{**})$, then 
\begin{equation*}
    J_T^{**}Q_T^{**}f=Q_T^{*}Q_T^{**}f=T_Nf,\qquad \mbox{and}\qquad  J_T^{**}Q_T^{**}f=J_T^{**}J_T^{*}f=T_Ff,
\end{equation*}
which imply $f\in \dom(T_N)\cap\dom (T_F)$. 
On the converse, if $f\in \dom(T_N)\cap\dom (T_F^{1/2})$, then $f\in \dom Q_T^{**}$, hence $Q_T^{**}f=J_T^{*}f$ and $f\in \dom J_T^{**}J_T^*$ imply $f\in\dom (J_T^{**}J_T^{*})$. As above, 
\begin{equation*}
    J_T^{**}Q_T^{**}f=T_Ff=T_Nf.
\end{equation*}
From these, statements (a) and (b) already follow.
\end{proof}
Let us recall that a positive operator $T:\hil\to\hil$ is called \textit{bounded from below} if it satisfies
\begin{equation}\label{E:T-boundedbelow}
    \sip{Tf}{f}\geq \varepsilon \|f\|^2,\qquad f\in\dom T
\end{equation}
for some constant $\varepsilon>0$. 
\begin{theorem}
Suppose that the densely defined positive symmetric operator $T$ is bounded from below. Then $T^{**}=J_T^{**}Q_T^{**}$. 
\end{theorem}
\begin{proof}
From Lemma \ref{L:Lemma} (b) it is clear that $T\subset J_T^{**}Q_T^{**}$ and also that the latter operator is closed. Hence
\begin{equation*}
    T^{**}\subset J_T^{**}Q_T^{**}.
\end{equation*}
Assume towards a contradiction that the inclusion above is proper. Then there is a non-zero vector $g\in \dom (J_T^{**}Q_T^{**})$ that  is orthogonal to the graph $G(T)$ of $T$. Then 
\begin{align*}
    0=\sip{(f,Tf)}{(g,J_T^{**}Q_T^{**}g)}=\sip fg+ \sip{Tf}{J_T^{**}Q_T^{**}g}
\end{align*}
for every $f\in\dom T$. From that we infer  $J_T^{**}Q_T^{**}g\in \dom T^*$ and 
\begin{equation}\label{E:TJQg}
    T^*J_T^{**}Q_T^{**}g=-g.
\end{equation}
Observe now that  \eqref{E:T-boundedbelow} impies that $Q_T^{**}:\dom Q_T^{**}\to\mathcal{E}_T$ is invertible with everywhere defined bounded inverse. Indeed, it satisfies 
\begin{equation*}
    \|Q_Tf\|_{\mathcal E_T}^2=\sip{Tf}{f}\geq \varepsilon \|f\|^2,\qquad f\in\dom T,
\end{equation*}
and its range space contains the dense set $\ran T\subseteq \mathcal E_T$. 
Note also that we  have  $J_TQ_T=T$ due to \eqref{E:QT*} and \eqref{E:JT*}. We claim that 
\begin{equation}\label{E:T*=Q*J*}
    T^*=Q_T^{*}J_T^*.
\end{equation}
For let $k\in \dom T^*=\dom (J_TQ_T)^*$ and $f\in\dom T$, then 
\begin{equation*}
    \sip{J_T(Tf)}{k}=\sip{(J_TQ_T)Q_T^{-1}(Tf)}{k}=\sipe{Tf}{(Q_T^{-1})^*(J_TQ_T)^*k},
\end{equation*}
whence $k\in \dom J_T^{*}$ and $J_T^*k=(Q_T^{-1})^*(J_TQ_T)^*k$. This yields \eqref{E:T*=Q*J*}, indeed.

Putting now \eqref{E:T*=Q*J*} and \eqref{E:TJQg} together we obtain 
\begin{equation*}
    -\|g\|^2=\sip{Q_T^{*}J_T^*J_T^{**}Q_T^{**}g}{g}=\|J_T^{**}Q_T^{**}g\|^2,
\end{equation*}
whence $g=0$, a contradiction.
\end{proof}
\begin{corollary}
Let $T:\hil\to\hil$ be a densely defined positive symmetric operator which is bounded from below in the sense of \eqref{E:T-boundedbelow}. Then $\dom (T_N)\cap\dom (T_F^{1/2})=\dom T^{**}$ and   
\begin{equation*}
    G(T^{**})=G(T_N)\cap G(T_F).
\end{equation*}
In particular, if $T_N=T_F$ then $T$ is essentially self-adjoint.
\end{corollary}
We remark here that the self-adjoint extensions $T_1$ and $T_2$ of $T$ are called \textit{disjoint} if they satisfy $\dom T_1\cap\dom T_2=\dom T^{**}$. Using this wording, the above corollary can also be rephrased by saying that $T_F$ and $T_N$ are disjoint extensions of $T$ (cf. e.g. \cite{HassiSandoviciDeSnoo}*{Proposition 4.3}). 

\section{Extensions of the modulus square of a linear operator}
In this section we are going to apply the foregoing results to the `modulus square' operators $T^*T$ of a given densely defined linear operarator $T:\hil\to\kil$. (Here, and everywhere below, $\hil$ and $\kil$ denote real or complex Hilbert spaces.) Clearly, $T^*T$ is positive and symmetric, but  not necessarily self-adjoint. In \cite{SebTarcs2012} it was proved that $T^*T$ always has a positive self-adjoint extension, regardless of whether $T$ is closable or not. (In the former case, $T^*T^{**}$ is apparently a positive self-adjoint extension of $T^*T$.) Using the results of Section \ref{S:section2} we can provide a brief and simple proof of this fact as follows:
\begin{theorem}
If $T:\hil\to\kil$ is a densely defined linear operator, then $T^*T$ has a positive self-adjoint extension.  
\end{theorem}
\begin{proof}
By Theorem \ref{T:Theorem1} it suffices to prove that $\mathcal{D}_*(T^*T)$ is  dense in $\hil$. However, from the identity
\begin{equation*}
    \mathcal{D}_*(T^*T)=\set{g\in\hil}{\sup\set{\abs{\sip{T^*Th}{g}}^2}{h\in\dom (T^*T), \|Th\|^2\leq 1}<+\infty}
\end{equation*}
it is readily seen that $\dtt$ contains the dense set $\dom T$, hence $\dtt$ itself is dense. 
\end{proof}
In the next statement the issue of uniqueness of positive self-adjoint extendibility of $T^*T$ is treated. As it has been noticed in the preceding section, uniqueness may only occur when we have   $\dom (T^*T)^{\perp}=\{0\}$.  We will therefore only discuss that  non-degenerate case:
\begin{theorem}
Let $T:\hil\to\kil$ be a densely defined linear operator such that $\dom (T^*T)^{\perp}=\{0\}$ and let $T_0$ denote the restriction of $T$ to $\dom (T^*T)$. Then the following statements are equivalent:
\begin{enumerate}[label=\textup{(\roman*)}]
    \item $T^*T$ has a unique positive self-adjoint extension,
    \item $\ran(I+T^*T)^{\perp}\cap\ran T_0^*=\{0\}$.
\end{enumerate}
\end{theorem}
\begin{proof}
Assume first that $T^*T$ has a unique positive self-adjoint extension. Consider a vector $k\in\dom (T^*T)$ such that $T_0^*k\in \ran (I+T^*T)^\perp$, then 
\begin{align*}
    \sup&\set{\abs{\sip{f}{T_0^*k}}^2}{f\in\dom (T^*T), \|Tf\|^2\leq 1}\\&=\sup\set{\abs{\sip{Tf}{k}}^2}{f\in\dom (T^*T), \|Tf\|^2\leq 1}\leq \|k\|^2.
\end{align*}
Hence Theorem \ref{C:Krein-unique} implies $T^*k=0$ and therefore  (i) implies (ii). 

Assume now (ii) and consider a vector $g\in\ker(I+T^*T)^*\cap\dtt$. Then here is a constant $C>0$ such that  
\begin{equation*}
    \abs{\sip{T^*Tf}{g}}^2\leq C\|Tf\|^2,\qquad f\in\dom T^*T,
\end{equation*}
from which with $(T^*T)^*g=-g$ we obtain that 
\begin{equation}\label{E:adjointrange}
    \abs{\sip{f}{g}}^2\leq C\|T_0f\|^2,\qquad f\in\dom T_0.
\end{equation}
Inequality \eqref{E:adjointrange} expresses that the linear functional $\phi:\ran T_0\to\dupK$,
\begin{equation*}
    \phi(T_0f)\coloneqq \sip{f}{g},\qquad f\in\dom T_0
\end{equation*}
is well-defined and continuous. The Riesz representation theorem yields  a vector $k\in\kil$ such that 
\begin{equation*}
    \sip{T_0f}{k}=\sip{f}{g},\qquad f\in\dom T_0
\end{equation*}
which means that $k\in\dom T_0^*$ and $T_0^*k=g$. As a consequence we see that $g\in \ran T_0^*$ and thus $g=0$ by assumption (ii). By Corollary \ref{C:Krein-unique}, $T^*T$ has a unique positive self-adjoint extension. 
\end{proof}

If $T$ is closable, then the most natural positive self-adjoint extension of $T^*T$ is  $T^*T^{**}$. In what follows we examine the relationship of that operator with the extreme extensions of $T^*T$. To do so we recall  that a vector subspace $\mathcal D$ is called a \textit{core} for a closed linear operator $S$ if $\mathcal D\subseteq \dom S$ and 
\begin{equation*}
    \overline{G(S|_{\mathcal D})}=G(S).
\end{equation*}
In other words, the graph of the restriction of $S$ to $\mathcal D$ is dense in the graph of $S$.
\begin{theorem}
Let $T$ be a densely defined and closable operator between $\hil$ and $\kil$. Then the following statements are equivalent:
\begin{enumerate}[label=\textup{(\roman*)}]
    \item $\dom (T^*T)$ is a core for $T^{**}$,
    \item $T^*T^{**}$ is identical with the Friedrichs extension of $T^*T$.
\end{enumerate}
\end{theorem}
\begin{proof}
First note that $\dom (T^*T)$ is dense because it is the core for a densely defined closed operator. Hence the Friedrichs extension  of $T^*T$ exists. According to Theorem \ref{T:S=TF}, $T^*T^{**}=(T^*T)_F$ if and only if
\begin{equation}\label{E:6.2}
    \dom T^{**}\cap \ran(I+T^*T)^{\perp}=\{0\}.
\end{equation}
(Here we used the identities $\dom (T^*T^{**})^{1/2}=\dom T^{**}$ and $\ker(I+T^*T)^*=\ran(I+T^*T)^{\perp}$.)

Assume first that $\dom (T^*T)$ is a core for $T^{**}$. Consider a vector $g$ from the set on the left hand side of \eqref{E:6.2}. Then for every $f\in\dom T^*T$,
\begin{align*}
    0&=\sip{g}{(I+T^*T)f}=\sip gf+\sip{T^{**}g}{Tf},
\end{align*}
which means that $(g,T^{**}g)$ is orthogonal to the graph of $T|_{\dom T^{*}T}$. This proves \eqref{E:6.2}. 

Assume on the contrary that $T^*T^{**}$ is identical with the Friedrichs extension of $T^*T$, or equivalently that $T$ fulfills \eqref{E:6.2}. Take a vector $g\in\dom T^{**}$ such that $(g,T^{**}g)$ is orthogonal to $G(T|_{\dom T^{*}T})$. Then
\begin{equation*}
    0=\sip[\big]{(g,T^{**}g)}{(f,Tf)}=\sip gf+\sip{T^{**}g}{Tf}=\sip{g}{f+T^*Tf}
\end{equation*}
for every $f\in\dom (T^*T)$. By \eqref{E:6.2}, $(g,T^{**}g)=0$ which in turn means that $\dom (T^*T)$ is a core for $T^{**}$. 
\end{proof}
\begin{remark}
We notice that the `density assumption'  
\begin{equation*}
    \overline{G(T|_{\dom(T^*T)})}=G(T^{**})
\end{equation*}
already  involves the closability of $T$. In fact, for every densely defined linear operator $T$, the restricted operator $T_0\coloneqq T|_{\dom(T^*T)}$ is automatically closable. For let $(f_n)$ be a sequence from $\dom (T^*T)$ such that 
\begin{equation}
    f_n\to 0\qquad \mbox{and}\qquad T_0f_n\to k
\end{equation}
for some $k\in \kil$. Then $k\in\overline{\dom T^*}$ because $T_0f_n\in\dom T^*$. On the other hand, 
\begin{equation*}
    \sip{k}{g}=\limn\sip{T_0f_n}{g}=\limn\sip{f_n}{T^*g}=0
\end{equation*}
for every $g\in\dom T^*$, whence $k\in(\dom T^*)^\perp$. Thus $k=0$ and $T_0$ is closable, accordingly.
\end{remark}
\begin{corollary}\label{C:Gesztesy1}
Let $T$ be a densely defined linear operator such that 
\begin{equation*}
    \ran T\subseteq \dom T^*.
\end{equation*}
Then $T$ is closable and $T^*T^{**}$ is identical with the Friedrichs extension of $T^*T$.
\end{corollary}
If $S:\hil\to\hil$ is a symmetric operator, then $S^2$ is positive and symmetric. The following Corollary gives a formula for its Friedrichs extension  (cf. \cite{Gesztesy}*{Theorem 3.1}) and simultaneously corrects the false assertion of \cite{ReedSimon}*{p.181, Corollary}.
\begin{corollary}
Let $S:\hil\to\hil$ be a symmetric operator such that $\mathcal D\coloneqq \dom S^2$ is dense. Then
\begin{equation*}
    (S^2)_F=(S|_{\mathcal D})^*(S|_{\mathcal D})^{**}.
\end{equation*}
Furthermore, $(S^2)_F=S^*S^{**}$ if and only if $\dom S^2$ is core for $S^{**}$.
\end{corollary}
\begin{proof}
For brevity's sake let us introduce the restricted operator $T\coloneqq S|_{\mathcal D}$. It is clear that $T^*T=S^2$ and that $\ran T\subseteq \dom T^*$, therefore
\begin{equation*}
    (S^2)_F=(T^*T)_F=T^*T^{**},
\end{equation*}
according to Corollary \ref{C:Gesztesy1}.

If $\dom S^2$ is core for $S^{**}$ then $S^{**}=T^{**}$ and $T^*=S^*$ hence $(S^2)_F=S^*S^{**}$ according to the first part of the proof. Suppose on the converse that $S^*S^{**}$ is identical with the Friedrichs extension of $S^2$, then
\begin{equation*}
    S^*S^{**}=T^*T^{**}.
\end{equation*}
In particular, $\dom S^{**}=\dom T^{**}$ which implies $S^{**}=T^{**}$ because $T\subset S$. This in turn means that $\mathcal D$ is a core for $S^{**}$.  
\end{proof}
We close this section with a characterization of essentially self-adjoint operators by means of the square of the adjoint. For similar characterizations we refer the reader to \cites{SebTarcs2016, SebTarcs2019, SebTarcs2020}. 
\begin{lemma}
Let $T$ be a densely defined and closed linear operator between $\hil$ and $\kil$ and let $\D$ be a dense linear subspace of $\dom T$. Letting $S$ be a restriction of $T$ to  $\D$,  the following assertions are equivalent:
\begin{enumerate}[label=\textup{(\roman*)}]
    \item  $\D$ is core for $T$, i.e., $\overline S=T$,
    \item $\dom S^*T\subset \dom S^{**}$
    \item $\ker (I+S^*T)=\{0\}$.
\end{enumerate}
\end{lemma}
\begin{proof}
(i)$\Rightarrow$(ii): If $S^{**}=T$, then   $\dom S^*T=\dom S^{*}S^{**}\subseteq \dom S^{**}$.

(ii)$\Rightarrow$(iii): Assume that $\dom (S^*T)\subset \dom S^{**}$ and take a vector $g\in\ker (I+S^*T)$. Then 
\begin{equation*}
    0=\sip{g+S^*Tg}{g}=\|g\|^2+\sip{S^*Tg}{g}=\|g\|^2+\sip{Tg}{S^{**}g}=\|g\|^2+\|Tg\|,
\end{equation*}
hence $g=0$.

(iii)$\Rightarrow$(i): Assume finally that $S^{**}\neq  T$. Then there exists $0\neq g\in \dom T$ such that $(g,Tg)\in G(S)^{\perp}$, whence
\begin{equation*}
    0=\sip{(f,Sf)}{(g,Tg)}^{}_{G(T)}=\sip fg+\sip{Sf}{Tg}
\end{equation*}
for every $f\in\dom S$. This implies $Tg\in \dom S^*$ and $g+S^*Tg=0$, that is, $g\in\ker(I+S^*T)$. 
\end{proof}
\begin{corollary}
For a densely defined symmetric operator $S$, the following assertions are equivalent:
\begin{enumerate}[label=\textup{(\roman*)}]
    \item $S$ is essentially self-adjoint,
    \item $\dom (S^*)^2\subset \dom S^{**}$,
    \item $\ker ((I+(S^*)^2)=\{0\}$,
    \item $(S^*)^2$ is positive.
\end{enumerate}
\end{corollary}
\begin{proof}
Apply the preceding Lemma    with $T\coloneqq S^*$.
\end{proof}
\section{Extensions of bounded symmetric operators}
In this section we are going to analyze the bounded self-adjoint extensions of symmetric operators. The main goal is to reprove Krein's fundamental result according to which a contractive symmetric operator $S$ can always extend to an everywhere defined contractive self-adjoint operator $\widetilde S$. What is more, those extensions form an operator interval $[S_m,S_M]$. In contrast to Krein's approach our starting point is a result on bounded positive extendibility (see also \cites{Sebestyen83a, SebStoch91}) which in fact is an easy consequence of Theorem 1:


\begin{theorem}\label{T:Neumannext}
Let $\hil$ be a real or complex Hilbert space and let $T$ be a positive symmetric operator defined on a linear subspace $\D$ of $\hil$. Then the following statements are eqiuivalent: 
\begin{enumerate}[label=\textup{(\roman*)}]
    \item $T$ can be extended to a  bounded positive operator to $\hil$,
    \item $\dt=\hil$,
    \item there exists a constant $\gamma>0$ such that 
    \begin{equation*}
        \|Tf\|^2\leq \gamma\cdot  \sip{Tf}{f},\qquad (\forall f\in\D).
    \end{equation*}
\end{enumerate} 
In any case, the Krein-von Neumann extension $T_N$ of $T$ is bounded with $\|T_N\|\leq \gamma $.
\end{theorem}
\begin{proof}
(i)$\Rightarrow$(ii): Let $\widetilde T\in\bh$ be a bounded positive extension of $T$. Then 
\begin{equation*}
    \abs{\sip{Tf}{g}}^2=\abs{\sip{\widetilde Tf}{g}}^2\leq \sip{\widetilde Tf}{f}\sip{\widetilde Tg}{g}=\sip{Tf}{f}\sip{\widetilde Tg}{g}
\end{equation*}
for $f\in\D$ and $g\in\hil$. Hence $\D_*(T)=\hil$.

(ii)$\Rightarrow$(iii): Taking into account of identity $\D_*(T)=\dom J_T^*$, assumption (ii) means that $J_T^*$ is a bounded operator by the closed graph theorem. Consequently, $T_N=J_T^{**}J_T^*$ is a bounded positive extension of $T$ and thus 
\begin{equation*}
    \|Tf\|^2=\|T_Nf\|^2\leq \|T_N\| \sip{T_Nf}{f}=\|T_N\| \sip{Tf}{f},\qquad f\in \D,
\end{equation*}
so (iii) holds true with $\gamma =\|T_N\|$.

(iii)$\Rightarrow$(i): For every $f\in\D$ and $g\in \hil$ one has 
\begin{equation*}
    \abs{\sip{Tf}{g}}^2\leq \gamma \cdot \|g\|^2\cdot \sip{Tf}{f},
\end{equation*}
because of (iii). Consequently, $\D_*(T)=\hil$.
\end{proof}
The above theorem is also of key importance for the introduction of the concept of the 'shorted operator'  (see eg. \cite{PekarevSmulian}): if $T\in\bh$ is a positive operator and $\D$ is a (closed) linear subspace of $\hil$ then  \begin{equation*}
    T-(T|_\D)_N
\end{equation*}
is called shortening of $T$  to the subspace $\D$. (Here  $(T|_{\D})_N$ denotes the Krein-von Neumann extension of $T|_{\D}$.) Its characteristic properties are described in the following lemma: 
\begin{lemma}
Let $\hil$ be a real or complex Hilbert space and let $S,T\in\bh$ be positive operators such that $S\leq T$. If $\D\subseteq \hil$ is any linear subspace such that $\ran S\subseteq \D^\perp$, then 
\begin{equation*}
    S\leq T-(T|_{\D})_N.
\end{equation*}
\end{lemma}
\begin{proof}
According to the assumption posed on $S$ we have $\D\subset\ker S$, hence $T|_{\D}\subset T-S$. Consequently, $T-S\geq (T|_{\D})_N$. 
\end{proof}
As a consequence of the construction of the Krein-von Neumann extension one readily obtains a formula for the quadratic form of the shorted operator (see eg. \cites{Anderson1,Anderson2}):
\begin{lemma}\label{E:Lemma7.2}
Let $T:\D\to\hil$ be a linear operator possessing a bounded positive extension to $\hil$. Then every  bounded positive extension $S$ of $T$ satisfies 
\begin{equation}
    \|(S-T_N)^{1/2}h\|^2=\inf_{f\in \D} \sip{S(f+h)}{f+h},\qquad h\in\hil.
\end{equation}
\end{lemma}
\begin{proof}
In accordance with the construction of $T_N$ we have
\begin{align*}
    0=\inf_{f\in\D} \|J_T^*h+Tf\|^2_T&=\inf_{f\in\D}\{\sip{T_Nh}{h}+\sip{h}{Tf}+\sip{Tf}{h}+\sip{Tf}{f}\}\\
    &=\sip{T_Nh}{h}+\inf_{f\in\D}\{\sip{h}{Sf}+\sip{Sf}{h}+\sip{Sf}{f}\},
\end{align*}
hence 
\begin{equation*}
    -\sip{T_Nh}{h}=\inf_{f\in\D}\{\sip{Sh}{f}+\sip{Sf}{h}+\sip{Sf}{f}\}.
\end{equation*}
Thus 
\begin{align*}
    \|(S-T_N)^{1/2}h\|^2&=\sip{Sh}{h}-\sip{T_Nh}{h}\\
                    &=\sip{Sh}{h} +   \inf_{f\in\D}\{\sip{Sh}{f}+\sip{Sf}{h}+\sip{Sf}{f}\}\\
                    &=\inf_{f\in \D} \sip{S(f+h)}{f+h},
\end{align*}
as  claimed.
\end{proof}

In the next theorem we provide an explicit formula for the range space of the square root of the shorted operator. In addition of being a sharpening of Krein's formula (2.1) this result also plays a fundamental role in the proof of Krein's uniqueness criterion (Theorem 7.7):
\begin{theorem}\label{T:ranS-TN}
Let $T:\D\to\hil$ be a linear operator possessing a bounded positive extension to $\hil$. For every bounded positive extension $S$ of $T$ we have
\begin{equation}
    \ran(S-T_N)^{1/2}=\ran S^{1/2}\cap \D^\perp.
\end{equation}
\end{theorem}
\begin{proof}
For brevity's sake let us introduce notation $Z\coloneqq (S-T_N)^{1/2}$. 

It is clear that 
\begin{equation*}
    \|Zh\|^2\leq \|S^{1/2}h\|^2
\end{equation*}
for all $h\in\hil$, hence $\ran Z\subseteq \ran S^{1/2}$ follows by the Douglas factorization theorem \cite{Douglas}. On the otherhand, \begin{equation*}
    (S-T_N)f=0,\qquad f\in\D,
\end{equation*}
because $S$ and $T_N$ agree on $\D$. Consequently, $\D\subseteq \ker Z$ which in turn implies $\ran Z\subseteq \D^\perp$. This proves inclusion
\begin{equation*}
    \ran Z\subseteq \ran S^{1/2}\cap \D^\perp.
\end{equation*}
To show the opposite subspace inclusion take any $g\in\ran S^{1/2}\cap\D^\perp$, $g=S^{1/2}k$. For every $f\in\D$ and $h\in\hil$,
\begin{align*}
    \abs{\sip gh}^2&=\abs{\sip{g}{f+h}}^2=\abs{\sip{S^{1/2}k}{f+h}}^2\leq \|k\|^2\cdot\sip{S(f+h)}{f+h},
\end{align*}
so that 
\begin{equation*}
    \abs{\sip gh}^2\leq \|k\|^2\cdot \inf_{f\in\D}\sip{S(f+h)}{f+h}=\|k\|^2 \cdot \|Z^{1/2}h\|^2,
\end{equation*}
according to Lemma \ref{E:Lemma7.2}. Then by \label{Sebesty83}{asdf} we get $g\in \ran (Z^{1/2})^*=\ran Z^{1/2}$. 
\end{proof}

At this point we note that the operation '$S\mapsto S_N$' is not monotone, that  is, 
\begin{equation}\label{E:Neumannext-monoton}
    S|_{\D}\leq T|_{\D} \quad \nRightarrow \quad (S|_{\D})_N\leq (T|_{\D})_N,
\end{equation}
as the following counter example demonstrates. Consider the  Hilbert space $\hil\coloneqq \dupC\times \dupC$, and the one-dimensional subspace $\D\coloneqq \dupC\times\{0\}$ in it. Let us introduce the positive operators $S,T:\D\to \hil$ by letting
\begin{equation*}
    S(z,0)\coloneqq (z,0),\qquad T(z,0)\coloneqq (z,z), \qquad (z\in\dupC).
\end{equation*}
Clearly, both $S$ and $T$ can be extended to positive operators onto $\dupC^2$, and also $S\leq T$ on $\D$. Hence $T-S$ is a (bounded) positive symmetric operator whose self-adjoint extensions are of the form
\begin{equation*}
    \opmatrix{0}{1}{1}{t},\qquad t\in\dupR.
\end{equation*}

The counterexample above thus shows that the Neumann extension does not preserve the partial ordering. However, the condition $S\leq T$ together with an additional range space inclusion already implies the inequality $S_N\leq T_N$:
\begin{proposition}
Let $S,T:\D\to\hil$ be linear operators possessing bounded positive extensions. Then the following two statements are equivalent:
\begin{enumerate}[label=\textup{(\roman*)}]
    \item $S_N\leq T_N$,
    \item \begin{enumerate}[label=\textup{(\alph*)}]
        \item $\sip{Sf}{f}\leq \sip{Tf}{f}$, $(\forall f\in\D)$,
        \item $\ran S_N^{1/2}\subseteq \ran T_N^{1/2}$.
    \end{enumerate}
\end{enumerate}
\end{proposition}
\begin{proof}
It is clear that (i) implies (ii). For the converse, observe first that (ii) (b) is equivalent to range inclusion
\begin{equation*}
    \ran J_S^{**}\subseteq \ran J_T^{**}.
\end{equation*}
By the Douglas factorization theorem, there exists a bounded linear operator $D:\ehil_T\to\ehil_S$ such that $J_S^{**}=J_T^{**}D$. Equivalently, we have $J_S^*=D^*J_T^*$  
\begin{align*}
    \|D^*\|^2&=\sup_{f\in\D, \sip{Tf}{f}\leq 1} \sip{D^*(Tf)}{D^*(Tf)}_{\mathcal{E}_S}\\
    &=\sup_{f\in\D, \sip{Tf}{f}\leq 1} \sip{D^*J_T^*f}{D^*J_T^*f}_{\mathcal{E}_S}\\
    &=\sup_{f\in\D, \sip{Tf}{f}\leq 1} \sip{J_S^*f}{J_S^*f)}_{\mathcal{E}_S}\\
    &=\sup_{f\in\D, \sip{Tf}{f}\leq 1} \sip{Sf}{f)} \leq 1,
\end{align*}
because of (i) (b). Consequently we get 
\begin{equation*}
    S_N=J_T^{**}DD^*J_T^*\leq J_T^{**}J_T^*=T_N,
\end{equation*}
hence (ii) implies (i).
\end{proof}

\medskip

In the remaining of the section we are going to consider bounded self-adjoint extensions of symmetric operators. First of all we provide a short and simple proof of Krein's fundamental theorem   stating that a bounded symmetric operator always has a norm-preserving (bounded) self-adjoint extension.
\begin{theorem}\label{T:SmSM}
Let $\D\subseteq \hil$ be a linear subspace of $\hil$ and let $S:\D\to\hil$, $\|S\|=1$ be a symmetric operator. Then 
\begin{equation}\label{E:SmandSM}
    S_{m}\coloneqq (I+S)_N-I\qquad\mbox{and}\qquad S_M\coloneqq I-(I-S)_N
\end{equation}
are self-adjoint extensions of $S$ having norm $1$. Moreover, 
\begin{equation*}
    [S_m,S_M]=\set{\widetilde{S}\in\bh}{\widetilde{S}^*=\widetilde{S},\|\widetilde{S}\|=1, S\subset \widetilde{S}}.
\end{equation*}
\end{theorem}
\begin{proof}
First of all observe that the positive symmetric operators $I\pm S$ satisfy the following inequalities:
\begin{equation*}
    \|(I\pm S)f\|^2\leq 2\sip{(I\pm S)f}{f},\qquad f\in\D.
\end{equation*}
By Theorem \ref{T:Neumannext} the corresponding Krein-von Neumann extensions $(I\pm S)_N$ exist and and have norm at most $2$. Accordingly, both $S_m$ and $S_M$ in \eqref{E:SmandSM} are bounded self-adjoint extensions of $S$ having norm $1$. Now, if $\widetilde S\in\bh$ is any self-adjoint extension of $S$ having norm $1$, then $I\pm \widetilde S $ are positive extensions of $I\pm S$, respectively. Consequently, 
\begin{equation*}
    (I\pm S)_N\leq I\pm\widetilde S
\end{equation*}
by the minimality of the Krein-von Neumann extension. Thus $S_M\leq \widetilde S\leq S_m$. 

Take now any self-adjoint operator $\widetilde S\in\bh$ such that $S_m\leq \widetilde S\leq S_M$. Then $0\leq S_M-\widetilde S\leq S_M-S_m$, whence 
\begin{align*}
    \abs{\sip{(S_M-\widetilde S)f}{h}}^2\leq \sip{(S_M-S_m)f}{f}\sip{(S_M-S_m)h}{h}=0,\qquad f\in\D, h\in\hil, 
\end{align*}
because of the Cauchy-Schwarz inequality. Thus $\widetilde Sf=S_Mf=Sf$ so that   $\widetilde S$ extends $S$. Finally, we have $1=\|S\|\leq \|S_m\|\leq \|\widetilde S\|\leq \|S_M\|=1$ that completes the proof.
\end{proof}

We close the paper by reproving Krein's uniqueness formula norm preserving self-adjoint extensions. Our proof is based on Theorem \ref{T:ranS-TN} and essentially differs from the original proof of \cites{Krein12}. Another elegant argument may be found in \cite{HassiMalamud}*{Proposition 3.19}.
\begin{theorem}
The  bounded symmetric operator $S:\D\to\hil$, $\|S\|=1$ has a unique self-adjoint norm 1 extension to $\hil$ if and only if for every $g\in\D^\perp$ one has 
\begin{equation}\label{E:73}
    \sup\set{\abs{\sip{Sf}{g}}^2}{f\in\D, \|f\|^2-\|Sf\|^2\leq 1}=+\infty.
\end{equation}
\end{theorem}
\begin{proof}
Before starting the proof of the claimed equivalence, let us introduce the self-adjoint operator
\begin{equation}
    \widetilde{S}\coloneqq \frac{1}{2}(S_m+S_M).
\end{equation}
It is easy to check that $S\subset \widetilde{S}$ and that $\|\widetilde{S}\|=1$. We also remark that  
\begin{equation}
    (I-\widetilde{S})-(I-S)_N=S_M-\widetilde{S}=\widetilde{S}-S_m=(I+\widetilde{S})-(I+S)_N=\frac{S_M-S_m}{2}.
\end{equation}
Hence, according to Theorem \ref{T:ranS-TN},
\begin{equation}\label{E:ranSM-Sm}
    \ran (S_M-S_m)^{1/2}=\ran(I-\widetilde{S})^{1/2}\cap\D^\perp=\ran(I+\widetilde{S})^{1/2}\cap\D^\perp,
\end{equation}
and therefore clearly
\begin{equation}\label{E:76}
     \ran (S_M-S_m)^{1/2}=\ran(I-\widetilde{S})^{1/2}\cap\ran(I+\widetilde{S})^{1/2}\cap\D^\perp.
\end{equation}
In the light of Theorem \ref{T:SmSM},  $S$ has  a unique norm-one self-adjoint extension if, and only if $S_m\neq S_M$, that is, when the set on the right hand side of \eqref{E:76} consists only of the zero vector.

Assume first that $S_m\neq S_M$ and take any nonzero vector 
\begin{equation*}
    0\neq g\in \ran(I-\widetilde{S})^{1/2}\cap\ran(I+\widetilde{S})^{1/2}\cap\D^\perp.
\end{equation*}
Choose, accordingly, $k,l\in\hil$ such that 
\begin{equation*}
    g=(I+\widetilde{S})^{1/2}k=(I-\widetilde{S})^{1/2}l.
\end{equation*}
Using the identities 
\begin{equation*}
    (I-\widetilde{S}^2)^{1/2}=(I+\widetilde{S})^{1/2}(I-\widetilde{S})^{1/2}=(I-\widetilde{S})^{1/2}(I+\widetilde{S})^{1/2}
\end{equation*}
we calculate 
\begin{align*}
    2\sip{Sf}{g}&=2\sip{f}{\widetilde{S}g}=\sip{f}{(I+\widetilde{S})g}-\sip{f}{(I-\widetilde{S})g}\\
    &=\sip{f}{(I+\widetilde{S})(I-\widetilde{S})^{1/2}l}-\sip{f}{(I-\widetilde{S})(I+\widetilde{S})^{1/2}k}\\
    &=\sip{f}{(I-\widetilde{S}^2)^{1/2}(I+\widetilde{S})^{1/2}l-(I-\widetilde{S}^2)^{1/2}(I-\widetilde{S})^{1/2}k}\\
    &=\sip{(I-\widetilde{S}^2)^{1/2}f}{(I+\widetilde{S})^{1/2}l-(I-\widetilde{S})^{1/2}k}.
\end{align*}
Thus, with $\alpha\coloneqq \|(I+\widetilde{S})^{1/2}l-(I-\widetilde{S})^{1/2}k\|^2$ we get 
\begin{align*}
    4\abs{\sip{Sf}{g}}^2\leq \gamma \cdot \|(I-\widetilde{S}^2)^{1/2}f\|^2= \gamma\cdot[\|f\|^2-\|Sf\|^2] 
\end{align*}
proving that  the supremum in \eqref{E:73} is finite.

Let us suppose now that 
\begin{equation}\label{E:7.7}
    \abs{\sip{Sf}{g}}^2\leq \alpha\cdot [\|f\|^2-\|Sf\|^2],\qquad f\in\D
\end{equation}
holds for some non-zero $g\in\D^\perp$ and $\alpha>0$ (depending only on $g$). Consider again the self-adjoint operator $\widetilde{S}\coloneqq \frac12(S_m+S_M)$ and introduce the following semi-norm 
\begin{equation*}
    p(h)\coloneqq \|h\|^2-\|\widetilde Sh\|^2,\qquad h\in\hil.
\end{equation*}
From \eqref{E:7.7} it follows that 
\begin{equation*}
    \frac{1}{\sqrt \alpha} \abs{\sip{Sf}{g}}\leq p(f)\leq \|f\|,\qquad f\in \D. 
\end{equation*}
By the Hahn-Banach theorem combined with the Riesz representation theorem, there exists a (unique) vector $h_0\in\hil$, $\|h_0\|\leq \sqrt \alpha,$ such that 
\begin{equation*}
    \sip{Sf}{g}=\sip{f}{h_0},\qquad f\in\D,
\end{equation*}
and 
\begin{equation*}
    \abs{\sip{h}{h_0}}^2\leq \alpha\cdot[\|h\|^2-\|\widetilde Sh\|^2]=\alpha \cdot \|(I-\widetilde S^2)^{1/2}\|^2,\qquad h\in\hil. 
\end{equation*}
By  \cite{Sebestyen83}*{Theorem 1} we obtain that $h_0\in\ran (I-\widetilde S^2)^{1/2}$ and hence   
\begin{equation*}
    h_0\in \ran (I-\widetilde S)^{1/2}\cap \ran (I+\widetilde S)^{1/2}.
\end{equation*}
Observe on the other hand that 
\begin{equation*}
    \sip{f}{h_0-\widetilde Sg} =\sip{Sf}{g}-\sip{Sf}{g}=0,\qquad f\in\D,
\end{equation*}
thus 
\begin{equation*}
    h_0-\widetilde Sg\in \D^\perp.
\end{equation*}
Consequently,
\begin{equation*}
    (I-\widetilde S)g+h_0=g+(h_0-\widetilde Sg)\in \D^\perp
\end{equation*}
and
\begin{equation*}
   -(I+\widetilde S)g+h_0=(h_0-\widetilde Sg)-g\in \D^\perp, 
\end{equation*}
but also
\begin{equation*}
  (I-\widetilde S)g+h_0\in \ran (I-\widetilde S)+\ran(I-\widetilde S)^{1/2}\subseteq\ran(I-\widetilde S)^{1/2} ,  
\end{equation*}
and similarly, $-(I+\widetilde S)g+h_0\in \ran(I+\widetilde S)^{1/2}$. Summing up, we have
\begin{equation}
  \begin{cases}
   (I-\widetilde S)g+h_0\in  \ran(I-\widetilde S)^{1/2}\cap\D^\perp, \\ -(I+\widetilde S)g+h_0\in  \ran(I+\widetilde S)^{1/2}\cap\D^\perp.
  \end{cases}
 \end{equation}
 Using identities \eqref{E:ranSM-Sm} it follows that
 \begin{equation*}
     0\neq 2g=(I-\widetilde S)g+h_0-[-(I+\widetilde S)g+h_0]\in\ran (S_M-S_m)^{1/2},
 \end{equation*}
 which apparently implies $S_M\neq S_m$.
\end{proof}

\end{document}